\documentclass[12pt,a4paper]{article}

\usepackage[ngerman,english]{babel}
\usepackage[utf8]{inputenc}
\usepackage[T1]{fontenc}
\usepackage{amsmath}
\usepackage{amssymb}
\usepackage{amsthm}
\usepackage{mathtools}
\usepackage{paralist}
\usepackage{tikz}
\usepackage{graphicx}
\usepackage{hyperref}
\usepackage[capitalise]{cleveref}

\begin{document}

\theoremstyle{plain}
\newtheorem{theorem}{Theorem}[section]
\newtheorem{definition}[theorem]{Definition}
\newtheorem{lemma}[theorem]{Lemma}
\newtheorem{prop}[theorem]{Proposition}
\newtheorem{corollary}[theorem]{Corollary}
\newtheorem{conjecture}[theorem]{Conjecture}
\theoremstyle{remark}
\newtheorem{remark}[theorem]{Remark}
\newtheorem{example}[theorem]{Example}

\newcommand{\reg}{\mathrm{reg}}
\newcommand{\Ass}{\mathrm{Ass}}
\newcommand{\charakt}{\mathrm{char}}
\newcommand{\diag}{\mathrm{diag}}
\newcommand{\GL}{\mathrm{GL}}
\newcommand{\Tor}{\mathrm{Tor}}
\newcommand{\im}{\mathrm{im}}
\newcommand{\coker}{\mathrm{coker}}
\newcommand{\id}{\mathrm{id}}
\newcommand{\length}{\mathrm{length}}
\newcommand{\LM}{\mathrm{LM}}
\newcommand{\LT}{\mathrm{LT}}
\newcommand{\cone}{\mathrm{cone}}
\newcommand{\ord}{\mathrm{ord}}
\newcommand{\Quot}{\mathrm{Quot}}
\newcommand{\Spec}{\mathrm{Spec}}
\newcommand{\height}{\mathrm{ht}}
\newcommand{\rank}{\mathrm{rank}}
\newcommand{\Ann}{\mathrm{Ann}}
\newcommand{\reynolds}{\mathcal{R}}
\newcommand{\maxId}{\mathfrak{m}}
\newcommand{\maxIdn}{\mathfrak{n}}
\newcommand{\primId}{\mathfrak{p}}

\title{Arithmetic invariant rings of some irreducible complex reflection groups}
\author{David Mundelius \\ \small{Technische Universität München, Zentrum Mathematik - M11} \\ \small{Boltzmannstraße 3, 85748 Garching, Germany} \\ \small{\texttt{david.mundelius@tum.de}}}
\date{March 15, 2023}
\maketitle

\begin{abstract}
In this article it is determined which integral reflection representations of the symmetric groups and the primitive complex reflection groups of degree $2$ have rings of invariants which are isomorphic to polynomial rings.
\end{abstract}


\section*{Introduction}

By a famous result of Shephard and Todd \cite{st} the ring of invariants of every finite complex reflection group is isomorphic to a polynomial ring. This has originally been proved by classifying all irreducible finite complex reflection groups and determining the ring of invariants for each of these groups explicitly. Every such group is defined over some algebraic number field $K$ and can in fact be defined over the ring of integers $R$ of $K$; however, sometimes there are several pairwise inequivalent representations of these groups over $R$ which are all equivalent to the reflection representation over $K$, see \cite{craig, feit1, feit2}.

The purpose of this article is to determine for some of the irreducible complex reflection groups whether or not their invariant rings over the ring of integers $R$ is isomorphic to a polynomial ring over $R$ and if not, which elements of $R$ need to be invertible in order to obtain a ring of invariants which is isomorphic to a polynomial ring. We start with some general results  in \cref{SectionGeneral} which will be needed for the computations in the later sections; see \cite{ainvpseudo} for more general results regarding the question when a ring of invariants of a finite group over a Dedekind domain is a polynomial ring.

In \cref{SectionSymmetric} we study one infinite family of irreducible complex reflection groups which appears in the classification by Shephard and Todd, namely the symmetric groups. These are defined as reflection groups over the integers and have several pairwise inequivalent reflection representations over $\mathbb{Z}$ which have been classified by Craig \cite{craig}. For each of these integral representations we determine over which rings $R$ with $\mathbb{Z} \subseteq R \subseteq \mathbb{Q}$ the ring of invariants is a polynomial ring.

In \cref{SectionDegree2} we discuss another class of irreducible complex reflection groups, namely the primitive reflection groups of degree $2$, i.e. those reflection groups in $Gl_2(\mathbb{C})$ which cannot be given by monomial matrices. There are 19 such groups which appear with numbers 4 to 22 in the classification by Shephard and Todd. Their integral reflection representations have been classified by Feit \cite{feit2}. For each of them we determine the invariant ring over $\mathbb{C}$ and then decide whether the ring of invariants over the respective ring of integers is a polynomial ring.

The explicit computations have been done with the computer algebra system \texttt{Singular} \cite{singular} and its library \texttt{finvar} \cite{finvar}. Many of the results of this article also appeared in the author's unpublished Master's thesis \cite{masterthesis}.

If $L$ is a finitely generated free $R$-module of rank $n$ for some commutative ring $R$, then we can view a polynomial in $R[x_1, \ldots, x_n]$ as a function on $L$ and write $R[L]$ instead of $R[x_1, \ldots, x_n]$. If $S$ is a commutative $R$-algebra, then we often write $S[L]$ instead of $S[L \otimes_R S]$.

\section*{Acknowledgements}

I want to thank my advisor Gregor Kemper for proposing this topic and for his continuous support. Further I acknowledge the support from the graduate program TopMath of the Elite Network of Bavaria and the TopMath Graduate Center of TUM Graduate School at Technische Universität München.

\section{General results} \label{SectionGeneral}

\begin{prop} \label{AlgIndLM}
Let $R$ be an integral domain with quotient field $K$ and let $f_1, \ldots, f_n \in R[x_1, \ldots, x_n]$ be homogeneous. Assume that there is a monomial ordering on $K[x_1, \ldots, x_n]$ such that the leading coefficients of $f_1, \ldots, f_n$ are all units in $R$ and the leading monomials of $f_1, \ldots, f_n$ are algebraically independent. Then we have $R[f_1, \ldots, f_n]=K[f_1, \ldots, f_n] \cap R[x_1, \ldots, x_n]$.
\end{prop}

\begin{proof}
We denote the monomial ordering by $\leq$ and the leading monomial of a polynomial $f$ with respect to $\leq$ by $\LM(f)$. Without loss of generality we may assume that the leading coefficients of all $f_i$ are $1$. Assume that there is a $p \in K[f_1, \ldots, f_n] \cap R[x_1, \ldots, x_n]$ which is not in $R[f_1, \ldots, f_n]$ and choose such a $p$ with minimal leading monomial. Then we can write $p=\sum_{i=1}^m c_i f_1^{e_{i,1}} \cdots f_n^{e_{i,n}}$ with $c_i \in K$ and $e_{i,j} \in \mathbb{N}$. We write $g_i=f_1^{e_{i,1}} \cdots f_n^{e_{i,n}}$ and we may assume that the $g_i$ are pairwise distinct. Since the leading monomials of $f_1, \ldots, f_n$ are algebraically independent, the $\LM(g_i)=\LM(f_1)^{e_{i,1}} \cdots \LM(f_n)^{e_{i,n}}$ are pairwise distinct and we may assume that $\LM(g_1)>\LM(g_2)> \ldots > \LM(g_n)$. Then $LM(p)=\LM(g_1)$ and the leading coefficient of $p$ is $c_1$. Since $p \in R[x_1, \ldots, x_n]$ this meams that $c_1 \in R$, so $c_1 g_1 \in R[f_1, \ldots, f_n]$ and for $p' \coloneqq p-c_1 g_1$ we have $\LM(p')<\LM(p)$ and $p' \in K[f_1, \ldots, f_n] \cap R[x_1, \ldots, x_n]$. By the choice of $p$, this implies that $p' \in R[f_1, \ldots, f_n]$ and therefore also $p \in R[f_1, \ldots, f_n]$.
\end{proof}

\begin{prop} \label{AlgIndResidue}
Let $R$ be a Dedekind domain with $K=\Quot(R)$. Let $L \coloneqq R^n$ and let $G \subseteq Gl(L)$ be a finite subgroup such that $K[L]^G$ is a polynomial ring generated by homogeneous elements $f_1, \ldots, f_n$. Then the following statements are equivalent:
\begin{compactenum}[(i)]
\item $R[f_1, \ldots, f_n] = R[L]^G$.
\item For every maximal ideal $\maxId \subset R$ we have $R_\maxId[f_1, \ldots, f_n] = R_\maxId[L]^G$.
\item For every maximal ideal $\maxId \subset R$ the classes of $f_1, \ldots, f_n$ are algebraically independent over the field $R/\maxId$.
\end{compactenum}
\end{prop}

\begin{proof}
It is clear that $(i)$ implies $(ii)$. For the implication $(ii) \implies (i)$ first note that we have $R[L]=\bigcap_\maxId R_\maxId[L]$ where the intersection ranges over all maximal ideals $\maxId \subset R$. This implies that $R[L]^G=\bigcap_\maxId R_\maxId[L]^G$ and we also have $R[f_1, \ldots, f_n]=\bigcap_\maxId R_\maxId[f_1, \ldots, f_n]$, so the claim follows.

The equivalence of $(ii)$ and $(iii)$ is proved in \cite[Lemma 3.3]{ainvpseudo}.
\end{proof}

The following result of Kemper \cite[Proposition 16]{kemperrefl} will be used several times within this article:

\begin{theorem} \label{KemperAlgInd}
Let $G$ be a finite group and let $V$ be a representation of $G$ over a field $K$. If the representation is faithful, then for elements $f_1, \ldots, f_n \in K[V]^G$ the following two properties are equivalent:
\begin{compactenum}[(i)]
\item The polynomials $f_1, \ldots, f_n$ are algebraically independent and we have $\deg f_1 \cdots \deg f_n=|G|$.
\item $K[V]^G=K[f_1, \ldots, f_n]$.
\end{compactenum}
\end{theorem}

\section{The symmetric groups} \label{SectionSymmetric}

Fix an integer $n \geq 3$. The symmetric group $S_n$ has an irreducible reflection representation over $\mathbb{Z}$ of rank $n-1$ which can be defined as follows: we consider the action of $S_n$ on $M \coloneqq \mathbb{Z}^n$ by permutation matrices and then restrict it to the $S_n$-invariant subspace $L_0 \coloneqq \{ (x_1, \ldots, x_n) \in M | \sum_{i=1}^n x_i =0 \}$. Using this, $L_0$ becomes a faithful $\mathbb{Z}$-representation of $S_n$.

\begin{theorem} \label{SymmL0}
The ring of invariants $\mathbb{Z}[L_0]^{S_n}$ is a polynomial ring.
\end{theorem}

\begin{proof}
Let $f_1, \ldots, f_n \in \mathbb{Z}[M]$ be the elementary symmetric polynomials with $\deg(f_i)=i$. Then it is well-known that $\mathbb{Z}[M]^{S_n}=\mathbb{Z}[f_1, \ldots, f_n]$ We claim that $\mathbb{Z}[L_0]$ is generated by the restrictions of $f_2, \ldots, f_n$ ro $L_0$, which we call $g_2, \ldots, g_n$.

By Proposition \ref{AlgIndResidue} it is sufficient to show that for every prime number $p$ the classes $\overline{g_2}, \ldots, \overline{g_n}$ of $g_2, \ldots, g_n$ in $\mathbb{F}_p[L_0]$ are algebraically independent.

We have $\mathbb{F}_p[L_0] \cong \mathbb{F}_p[M]/(\overline{f_1})$, so it is sufficient to show that the classes of $f_2, \ldots, f_n$ in $\mathbb{F}_p[M]/(\overline{f_1})$ are algebraically independent. But these classes generate the $\mathbb{F}_p$-subalgebra $\mathbb{F}_p[f_1, \ldots, f_n]/(f_1)$ which has Krull dimension $n-1$ and therefore they must be algebraically independent.
\end{proof}

However, there are several inequivalent representation of $S_n$ over $\mathbb{Z}$ which are equivalent to $L_0$ over $\mathbb{Q}$; these have been classified by Craig \cite[Appendix 2]{craig}. In order to give them, we need some notation: let $E^{i,j}$ be the $(n-1) \times (n-1)$-matrix for which the $(i,j)$-th entry is $1$ and all other entries are $0$. By $E^{0,j}$ or $E^{n,j}$ we always mean the zero matrix. Now for all $1 \leq k \leq n-1$ we define $F_k \coloneqq I_{n-1}+E^{n-k-1,n-k}-2E^{n-k,n-k}+E^{n-k+1,n-k}$ where $I_{n-1}$ denotes the identity matrix. Moreover for every $d \in \mathbb{N}$ which divides $n$, we set $V_d \coloneqq d E^{1,1} + \sum_{i=2}^{n-1} (E^{i,i} + (n-i) E^{1,i}) \in \GL_{n-1}(\mathbb{Q})$. Now we can formulate Craig's theorem:

\begin{theorem} \textup{(Craig \cite{craig})}
For each divisor $d$ of $n+1$ one can define a matrix representation of $S_{n+1}$ over $\mathbb{Z}$ of rank $n-1$ by mapping the transposition $(k,k+1)$ for $1 \leq k \leq n-1$ to $V_d^{-1} F_k V_d$. These representations are pairwise inequivalent over $\mathbb{Z}$, but over $\mathbb{Q}$ they are all equivalent to the representation $L_0$ defined above. Moreover, any $\mathbb{Z}$-representation of $S_n$ which is equivalent to $L_0$ over $\mathbb{Q}$ is equivalent to one of these representations over $\mathbb{Z}$.
\end{theorem}

We will denote the $\mathbb{Z}$-representation given by the matrices $V_d^{-1} F_k V_d$ by $L_d$. By explicit calculations we find the following properties of the matrices used in the theorem:

\begin{lemma}  \ \label{FkExplicit}
\begin{compactenum}[a)]
\item For $1 \leq k \leq n-3$ we have $V_d^{-1} F_k V_d=F_k$.
\item The first row of $V_d^{-1} F_{n-2} V_d$ equals $(1,\frac{n}{d},0, \ldots, 0)$.
\item The first row of $V_d^{-1} F_{n-1} V_d$ equals $(1-n, -\frac{n}{d}(n-2), -\frac{n}{d}(n-3), \ldots, -\frac{n}{d})$.
\end{compactenum}
\end{lemma}

Craig's theorem tells us that the representation $L_0$ must be equivalent to one of the $L_d$ over $\mathbb{Z}$; it will become clear later that this is in fact $L_n$. 

\begin{lemma} \label{PolyNecess}
Assume that $n \geq 5$ and let $d>0$ be a divisor of $n$. If the ring of invariants $R[L_d]^{S_n}$ is a polynomial ring for some ring $R$ with $\mathbb{Z} \subseteq R \subseteq \mathbb{Q}$, then $\frac{d}{n} \in R$.
\end{lemma}

\begin{proof}
Assume that $\frac{d}{n} \notin R$, so there is a prime number $p$ which divides $\frac{n}{d}$ and is not a unit in $R$. Let $\overline{L_d}$ denote the $\mathbb{F}_p$-representation of $S_n$ which is given by reducing all the matrices defining $L_d$ modulo $p$. By \cref{FkExplicit} we find that these reduced matrices for the identity and the transpositions $(1,2)$ and $(2,3)$ are pairwise distinct, so the image of $S_n$ in $Gl_{n-1}(\mathbb{F}_p)$ consists of at least three elements. For $n \geq 5$ this implies that the map $S_n \to Gl_{n-1}(\mathbb{F}_p)$ must be injective since otherwise $A_n$ would be contained in the kernel.

Since $R[L_d]^{S_n}$ is a polynomial ring it is generated by homogeneous elements of degrees $2, \ldots, n$ which also generate the invariant ring over $\mathbb{Q}=\Quot(R)$. The classes of these elements over $\mathbb{F}_p$ are therefore algebraically independent by Proposition \ref{AlgIndResidue} and hence generate the invariant ring $\mathbb{F}_p[\overline{L_d}]^{S_n}$ by \cref{KemperAlgInd} since we have seen above that the representation $\overline{L_d}$ is faithful.

In particular, this shows that $\mathbb{F}_p[\overline{L_d}]^{S_n}$ contains no element of degree $1$. \cref{FkExplicit} shows that all transpositions $(k,k+1)$ with $1 \leq k <n$ act on $\overline{L_d}$ by matrices in which the first row equals $(1, 0, \ldots, 0)$. Therefore all these elements fix $x_1 \in \mathbb{F}_p[\overline{L_d}]$ and since these transpositions generate $S_n$ it follows that $x_1 \in \mathbb{F}_p[\overline{L_d}]^{S_n}$, a contradiction.
\end{proof}

\cref{PolyNecess} shows that $\mathbb{Z}[L_d]^{S_n}$ is not a polynomial ring for $d \neq n$. Together with \cref{SymmL0}, this implies that $L_0$ cannot be equivalent to any of these, so it must be equivalent to $L_n$. In particular, $\mathbb{Z}[L_n]^{S_n}$ is a polynomial ring. An explicit calculation shows that $V_d^{-1} V_n=\diag(\frac{n}{d}, 1, \ldots, 1)$, so if $R$ is a ring with $\mathbb{Z} \subseteq R \subseteq \mathbb{Q}$ and $\frac{d}{n} \in R$, then $L_d$ and $L_n$ are equivalent over $R$, so $R[L_d]^{S_n}$ is a polynomial ring. Together with \cref{PolyNecess}, this yields the following result.

\begin{theorem}
Let $R$ be a ring with $\mathbb{Z} \subseteq R \subseteq \mathbb{Q}$ and let $n \geq 5$ be an integer and $d>0$ a divisor of $n$. Then $R[L_d]^{S_n}$ is a polynomial ring if and only if $\frac{d}{n} \in R$.
\end{theorem}

It remains to consider the cases $n \in \{ 3,4 \}$. We first consider the homomorphism $S_3 \to Gl_2(\mathbb{F}_3)$ given by the representation $L_1$ of $S_3$ and reduction modulo $3$. An explicit calculation shows that the images of the identity and the transpositions $(1,2)$ and $(2,3)$ in $Gl_2(\mathbb{F}_3)$ are pairwise distinct. Since $S_3$ contains no normal subgroup of order $2$, this implies that the homomorphism is injective and therefore with the same argument as in the proof of \cref{PolyNecess} we obtain that $R[L_1]^{S_3}$ con only be a polynomial ring if $\frac{1}{3} \in R$. Since $R[L_0]^{S_3}$ is a polynomial ring by \cref{SymmL0}, $L_0$ must again be $\mathbb{Z}$-equivalent to $L_3$, so $\mathbb{Z}[L_3]^{S_3}$ is a polynomial ring and $R[L_1]^{S_3}$ is a polynomial ring if and only if $\frac{1}{3} \in R$.

For $S_4$ we first consider the homomorphism $S_4 \to Gl_3(\mathbb{F}_2)$ given by the representation $L_1$ and reduction modulo $2$. An explicit calculation here shows that $(1,2)(3,4)$ is not in the kernel of this map, so the homomorphism must be injective since $(1,2)(3,4)$ is contained in every non-trivial normal subgroup of $S_4$. Again by using the same argument as in the proof of \cref{PolyNecess} we find that $R[L_1]^{S_4}$ can only be a polynomial ring if $\frac{1}{2} \in R$. For the representations $L_2$ and $L_4$, the corresponding homomorphisms $S_4 \to Gl_3(\mathbb{F}_2)$ are not injective, so instead we compute explicit generators for the invariant rings over $\mathbb{Q}$ using \texttt{Singular}. If we identify $\mathbb{Q}[L_4]$ with $\mathbb{Q}[x,y,z]$, we find $\mathbb{Q}[L_4]^{S_4}=\mathbb{Q}[f,g,h]$ where
\begin{align*} 
f:&=6x^2+8xy+3y^2+4xz+3yz+z^2, \\
g:&=8x^3+16x^2y+10xy^2+2y^3+8x^2z+10xyz+3y^2z+2xz^2+yz^2, \\
h:&=3x^4+8x^3y+7x^2y^2+2xy^3+4x^3z+7x^2yz+3xy^2z+x^2z^2+xyz^2.
\end{align*}
By appling Proposition \ref{AlgIndLM} with the lexicographic ordering with $z>y>x$ we obtain that $\mathbb{Z}[L_4]^{S_4}=\mathbb{Z}[f,g,h]$ is a polynomial ring.

For $L_2$ we also first use \texttt{Singular} to obtain generators for $\mathbb{Q}[L_2]^{S_4}$ and get $\mathbb{Q}[L_2]^{S_4}=\mathbb{Q}[f',g',h']$ where
\begin{align*} 
f':&=3x^2+8xy+6y^2+4xz+6yz+2z^2, \\
g':&=x^3+4x^2y+5xy^2+2y^3+2x^2z+5xyz+3y^2z+xz^2+yz^2, \\
h':&=3x^4+16x^3y+28x^2y^2+16xy^3+8x^3z+28x^2yz+24xy^2z+4x^2z^2+8xyz^2.
\end{align*}

Again Proposition \ref{AlgIndLM} with the lexicographic ordering with $z>y>x$ gives that $\mathbb{Z}[\frac{1}{2}][L_2]^{S_4}=\mathbb{Z}[\frac{1}{2}][f',g',h']$, but $f',g',h'$ do not generate the invariant ring over $\mathbb{Z}$ since $k'=\frac{1}{4}(f'^2+h')$ has coefficients in $\mathbb{Z}$. But clearly we have $\mathbb{Z}[\frac{1}{2}][f',g',h']=\mathbb{Z}[\frac{1}{2}][f',g',k']$ and by using \texttt{Singular} we find that $f',g',k'$ are algebraically independent modulo $2$, so by Proposition \ref{AlgIndResidue} we get that $\mathbb{Z}[L_2]^{S_4}=\mathbb{Z}[f',g',k']$.

So we can summarize the results of this section as follows:

\begin{theorem}
Let $R$ be a ring with $\mathbb{Z} \subseteq R \subseteq \mathbb{Q}$ and let $n \geq 3$ be an integer and $d>0$ a divisor of $n$. Then $R[L_d]^{S_n}$ is a polynomial ring if and only if $\frac{d}{n} \in R$ except in the case $n=4$ and $d=2$ where $\mathbb{Z}[L_2]^{S_4}$ is also a polynomial ring.
\end{theorem}

\section{The primitive groups of degree 2} \label{SectionDegree2}

In this section we discuss the arithmetic invariant rings of the primitive pseudoreflection groups of degree $2$, that is, the groups number 4 to 22 in the classification by Shephard and Todd \cite{st}. For each of these groups, the smallest number field over which they are defined together with the reflection representations over the ring of integers of this number field can be found in \cite{feit2}. In order to simplify the arguments, we will often work over a number field that is larger than necessary. The main step in order to see that this does not change the results is the following proposition.

\begin{prop} \label{PolyFieldExt}
Let $R$ be a principal ideal domain with $K=\Quot(R)$, let $L/K$ be a finite field extension and let $S$ be the integral closure of $R$ in $L$. Moreover, let $G \subseteq Gl_n(R)$ be a finite group. If the ring of invariants $S[x_1, \ldots, x_n]^G$ is isomorphic to a polynomial ring, then $R[x_1, \ldots, x_n]^G$ is also isomorphic to a polynomial ring.
\end{prop}

\begin{proof}
Let $A \coloneqq R[x_1, \ldots, x_n]$ and $B \coloneqq S[x_1, \ldots, x_n]$. Since $R$ is a principal ideal domain and $S$ is finitely generated as an $R$-module, $S$ is in fact a free $R$-module; let $\{ s_1, \ldots, s_k \}$ be a basis. Then $s_1, \ldots, s_k$ also form a basis of $B$ as an $A$-module. For $f=f_1 s_1 + \ldots + f_k s_k \in B$ with $f_1, \ldots, f_k \in A$ and $\sigma \in G$ we have $\sigma(f)=\sigma(f_1) s_1 + \ldots + \sigma(f_k) s_k$ since $G$ acts trivially on $S$. So if $f \in B^G$ then we must have $f_1, \ldots, f_k \in A^G$ and therefore $s_1, \ldots, s_k$ generate $B^G$ as an $A^G$-module; in particular, $B^G$ is a free $A^G$-module.

Now let $P \subset A^G$ by a prime ideal and choose a prime ideal $Q \subset B^G$ with $P=A^G \cap Q$.
Since $B^G$ is a regular ring by assumption, the same holds for $(B^G)_Q$. Moreover, since $B^G$ is free over $A^G$, $(B^G)_Q$ is a flat $(A^G)_P$-module. Therefore, $(A^G)_P$ is also regular, see \cite[Theorem 2.2.12]{bh}. Since this holds for every $P$, we obtain that $A^G$ itself is regular. By \cite[Corollary 2.2]{ainvpseudo} it follows that $A^G$ is isomorphic to a polynomial ring over $R$.
\end{proof}

In order to apply this proposition to the groups discussed below, we need to note two additional facts. First, we need to check that for each group the field over which it is defined has class number $1$. This is indeed the case as mentioned in \cite{feit2}. Second, we need to check that over the larger fields we want to consider there are no additional integral reflection representations; this follows directly from the arguments given in \cite{feit2}.

Unless something else is said explicitly, in every application of Proposition \ref{AlgIndLM} in this section we use the lexicographic ordering with $x > y$.

\subsection{The groups number 8 to 15}

The definition of these groups is based on the following two matrices:

\[S=\frac{1}{\sqrt{2}}\begin{pmatrix} -1  & i \\ -i & 1 \end{pmatrix}, T= \frac{1}{\sqrt{2}} \begin{pmatrix} \varepsilon & \varepsilon \\ \varepsilon^3 & \varepsilon^7 \end{pmatrix}. \]

Here $\varepsilon=e^\frac{i \pi}{4}$ denotes a primitive eighth root of unity. Now one matrix reflection representation of each of the different groups can be defined as follows, where $\omega=e^\frac{2 \pi i}{3}$; see \cite[Section 4]{st}:
\begin{align*}
G_8 &=\langle \varepsilon S, T \rangle \\
G_9 &=\langle S, \varepsilon T \rangle \\
G_{10} &=\langle \varepsilon^5 \omega^2 S, -\omega T \rangle \\
G_{11} &=\langle S, \varepsilon \omega  T \rangle \\
G_{12} &=\langle S, T \rangle \\
G_{13} &=\langle \varepsilon S, iT \rangle \\
G_{14} &=\langle \varepsilon S, -\omega T \rangle \\
G_{15} &=\langle \varepsilon S, i \omega T \rangle
\end{align*}

From \cite{feit2} we know that each of these groups has only one integral reflection representation which we denote by $L$. An explicit calculation shows that this can be obtained by conjugating the generators given above with
\[ U = \frac{1}{\sqrt{2}} \begin{pmatrix} 1 & 1 \\ i & 1 \end{pmatrix}. \]
By Proposition \ref{PolyFieldExt} we may compute all the rings of invariants over $R=\mathbb{Z}[\varepsilon, \omega]$ and $K=\Quot(R)$ although many of the groups may be defined over smaller fields. We start with the smallest of the groups, namely $G_{12}$; using \texttt{Singular} we find the following two invariants:
\begin{align*}
f &=(i+1)x^5y-5x^4y^2+5(1-i)x^3y^3+5ix^2y^4-(i+1)xy^5 \\
g &=x^8+4(i-1)x^7y-14ix^6y^2+14(i+1)x^5y^3-21x^4y^4+14(1-i)x^3y^5 \\ &+14ix^2y^6-4(i+1)xy^7+y^8
\end{align*}
Using \cref{KemperAlgInd} we find $K[L]^{G_{12}}=K[f,g]$. One can check easily that $f$ and $g$ are algebraically independent modulo $i+1$, so by Propositions \ref{AlgIndLM} and \ref{AlgIndResidue} we get $R[L]^{G_{12}}=R[f,g]$. For most of the other groups we now find the rings of invariants easily: the description of the groups above shows that certain powers of $f$ and $g$ are invariant which generate the invariant ring over $K$ by \cref{KemperAlgInd} and then also over $R$ by the same arguments as above.
\begin{align*}
R[L]^{G_9}&=R[f^4,g] \\
R[L]^{G_{11}}&=R[f^4,g^3] \\
R[L]^{G_{13}}&=R[f^2,g] \\
R[L]^{G_{14}}&=R[f,g^3] \\
R[L]^{G_{15}}&=R[f^2,g^3]
\end{align*}

For $G_8$ and $G_{10}$ the situation is a little bit more complicated: for $G_8$ we find $f^4$ and $g$ as invariants, but $\deg(f^4) \cdot \deg(g)=192 \neq 96 = |G_8|$, so there must be another invariant. Indeed we find
\begin{align*}
h &=2x^{12}+12(i-1)x^{11}y-66ix^{10}y^2+110(i+1)x^9y^3-231x^8y^4+132(1-i)x^7y^5\\ &+132(i+1)x^5y^7-231x^4y^8+110(1-i)x^3y^9+66ix^2y^{10}-12(i+1)xy^{11}+2y^{12} \in R[L]^{G_8}.
\end{align*}

We have $K[L]^{G_8}=K[g,h]$ and since $g$ and $h$ are the only invariants up to constant factors of their respective degrees, $R[L]^{G_8}$ must be $R[g,h]$ if it is a polynomial ring. However, we have $f^4=\frac{1}{27}(h^2-4g^3) \notin R[g,h]$, so $R[L]^{G_8}$ is not a polynomial ring. However, since we know that $f$ and $g$ are algebraically independent modulo every prime, the same arguments as above show that the same holds for $g$ and $h$ for all primes not dividing $2$ and $3$ and by an explicit calculation we see that they are also algebraically independent modulo primes dividing $2$ and hence Proposition \ref{AlgIndResidue} shows that $(R[\frac{1}{3}])^{G_8}$ is a polynomial ring over $R[\frac{1}{3}]$.

For $G_{10}$ we find $K[L]^{G_{10}}=K[f^4,h]=K[g^3,h]$. However using $f^4=\frac{1}{27}(h^2-4g^3)$ we find that $t \coloneqq g^3+7f^4=\frac{1}{27}(-g^3+7h^2)=\frac{1}{4}(h^2-f^4)$ is in $R[L]^{G_{10}}$ but neither in $R[f^4,h]$ nor in $R[g^3,h]$. Moreover, one can check with similar arguments as above that $t$ and $h$ are algebraically independent modulo every prime in $R$ and therefore $R[L]^{G_{10}}=R[h,t]$.

\subsection{The groups number 16 to 22}

The approach for these groups is similar to the previous section. Let $\zeta \coloneqq e^{\frac{i \pi}{30}}$ be a primitive $60$-th root of unity. Then all the groups can be defined over the field $K=\mathbb{Q}(\zeta)$ with ring of integers $R=\mathbb{Z}[\zeta]$. We write $\eta=\zeta^{12}$ and $\omega=\zeta^{20}$. For the definition of the groups we use the following two matrices:
\[ S=\frac{1}{\sqrt{5}}\begin{pmatrix} \eta^4-\eta & \eta^2-\eta^3 \\ \eta^2-\eta^3 & \eta-\eta^4 \end{pmatrix}, T=\frac{1}{\sqrt{5}} \begin{pmatrix} \eta^2-\eta^4 & \eta^4-1 \\ 1-\eta & \eta^3-\eta \end{pmatrix}. \]
With these we can define the groups as follows, see \cite[Section 4]{st}:
\begin{align*}
G_{16} &=\langle -\eta^3 S, T \rangle \\
G_{17} &=\langle i S, i \eta^3 T \rangle \\
G_{18} &=\langle - \omega \eta^3 S, \omega^2  T \rangle \\
G_{19} &=\langle i \omega S, i \eta^3 T \rangle \\
G_{20} &=\langle S, \omega^2 T \rangle \\
G_{21} &=\langle i S, \omega^2 T \rangle \\
G_{22} &=\langle i S, T \rangle
\end{align*}
Again it is proved in \cite{feit2} that each of these groups has only one integral reflection representation defined over $R$ which we denote by $L$. It can be obtained by conjugating the matrices given above with
\[ U=\begin{pmatrix} \eta^3+\eta^2+2 \eta +1 & 0 \\ \eta^3+\eta^2 & 1 \end{pmatrix}. \]

We first look at the ring of invariants $R[L]^{G_{22}}$. Using \texttt{Singular} we find two possible generators in this ring: one invariant $f$ of degree $12$ with leading term $px^{11}y$ where $p=3 \eta^3-\eta^2+2\eta+1 \in R$ and one invariant $g$ of degree $20$ with leading term $x^{20}$. The element $p$ satisfies $p^4=5\alpha$ where $\alpha=55\eta^3+55\eta^2+89$ is a unit in $R$. From \cref{KemperAlgInd} we obtain that $K[L]^{G_{22}}=K[f,g]$. An explicit computation shows that $f$ and $g$ are algebraically independent modulo $p$, so by Propositions \ref{AlgIndLM} and \ref{AlgIndResidue} we find $R[L]^{G_{22}}=R[f,g]$. In the same way we also get $R[L]^{G_{17}}=R[f^5,g]$, $R[L]^{G_{19}}=R[f^5,g^3]$, $R[L]^{G_{21}}=R[f,g^3]$.

Next we look at $G_{20}$. Here using \texttt{Singular} we find $f$ and a new invariant $h$ of degree $30$ as generators for the invariant ring over $K$. The leading term of $h$ is $5px^{30}$ and as before we obtain $R([\frac{1}{p}])[L]^{G_{20}}=(R[\frac{1}{p}])[f,h]$. However, $f^5-h^2$ is divisible by $p$, so they cannot be generators over $R$. Since there are no other invariants of the same degrees as $f$ and $h$, $R[L]^{G_{20}}$ cannot be a polynomial ring.

For $G_{16}$ we have $K[L]^{G_{16}}=K[g,h]$ and these are the only possible generators. An explicit calculation shows that $k \coloneqq \frac{1}{1728}(h^2-25p^2g^3)$ has coefficients in $R$, so every subring of $K$ over which the ring of invariants is a polynomial ring must contain $R[\frac{1}{1728}]=R[\frac{1}{6}]$. The leading term of $k$ is $5 p \beta x^{55} y^5$ where $\beta=5 \eta^3+5 \eta^2+8 \in R^\times$. Therefore by  Proposition  \ref{AlgIndLM} we have $R[\frac{1}{5}][g,k]=R[\frac{1}{5}][x,y] \cap K[g,k]$. Now let $F \in R[\frac{1}{5}][L]^{G_{16}}=K[g,h] \cap R[\frac{1}{5}][x,y]$. Then we have $F=F_1+hF_2$ where $F_1, F_2 \in K[g,h^2] \cap R[\frac{1}{5}][x,y]=K[g,k] \cap R[\frac{1}{5}][x,y]=R[\frac{1}{5}][g,k] \subseteq R[\frac{1}{30}][g,h^2]$. This shows that $R[\frac{1}{30}][L]^{G_{16}}=R[\frac{1}{30}][g,h]$. Finally, an explicit calculation shows that $g$ and $h$ are algebraically independent modulo $p$, so $R[\frac{1}{6}][L]^{G_{16}}=R[\frac{1}{6}][g,h]$ and this is the best result we can get.

Finally, for $G_{18}$ we have $K[L]^{G_{18}}=K[h,g^3]=K[h,k]$ where $k$ is defined as in the discussion of $G_{16}$ above. Using Proposition \ref{AlgIndLM} we get that the invariant ring over $R[\frac{1}{p}]$ is $R[\frac{1}{p}][L]^{G_{18}}=R[\frac{1}{p}][h,k]$. The polynomials $k$ and $h$ are not algebraically independent modulo $p$, but with $l \coloneqq \frac{(3p^6-2)h^2+k}{p^{10}}$ we have $R[\frac{1}{p}][h,k]=R[\frac{1}{p}][h,l]$ and an explicit calculation shows that $h$ and $l$ are algebraically independent modulo $p$, so $R[L]^{G_{18}}=R[h,l]$ is a polynomial ring.

\subsection{The groups number 4 to 7}

For these groups we have several integral reflection representations which are all equivalent over the complex numbers. We first define the groups as matrix groups for one of these representations as follows, see \cite[Section 3]{feit2}. All of them can be defined over the field $K=\mathbb{Q}(\omega,i)$ with ring of integers $R=\mathbb{Z}[\omega,i]$; here $\omega=e^{\frac{2 \pi i}{3}}$ denotes a primitive third root of unity.
We consider the matrices
\[ S=\begin{pmatrix} i & 0 \\ 0 & -i \end{pmatrix}, T=\begin{pmatrix} -\omega^2 & 1 \\ 0 & -\omega \end{pmatrix}. \]
The four groups are then defined as
\begin{align*}
G_4 &= \langle -S, -\omega T \rangle, \\
G_5 &= \langle -\omega S, -\omega T \rangle, \\
G_6 &= \langle iS, -\omega T \rangle, \\
G_7 &= \langle i \omega S, -\omega T \rangle
\end{align*}

We denote this representation of these groups as $L_1$. In order to determine their invariant rings, we need the following two polynomials where we write $p \coloneqq 2 \omega +1$:
\begin{align*}
f &\coloneqq px^4+4x^3y-2px^2y^2-4xy^3+py^4, \\
g &\coloneqq x^6-2px^5y-5x^4y^2-5x^2y^4+2pxy^5+y^6.
\end{align*}

Using \texttt{Singular} we obtain $K[L_1]^{G_4}=K[f,g]$ and therefore from the  definition of the groups and \cref{KemperAlgInd} we get $K[L_1]^{G_5}=K[f^3,g]$, $K[L_1]^{G_6}=K[f,g^2]$, $K[L_1]^{G_7}=K[f^3,g^2]$. The polynomial $g \coloneqq \frac{1}{64}(f^3-p^3g^2)$ has coefficients in $R$ and leading term $x^9y^3$. We get $K[L_1]^{G_5}=K[g,h]$ and $K[L_1]^{G_7}=K[g^2,h]$, so by Proposition \ref{AlgIndLM} we obtain $R[L_1]^{G_5}=R[g,h]$ and $R[L_1]^{G_7}=R[g^2,h]$, so both are polynomial rings. Moreover, $K[L_1]^{G_6}=K[f,h]$ and an explicit calculation shows that $f$ and $h$ are algebraically independent modulo $p$, so $R[L_1]^{G_6}=R[f,h]$ by Proposition \ref{AlgIndResidue}.

If $R[L_1]^{G_4}$ is a polynomial ring, then it must be $R[f,g]$, but $h \in R[L_1]^{G_4}$ and $h \notin R[f,g]$, so this is not the case. We show now that we have $R[\frac{1}{2}][L_1]^{G_4}=R[\frac{1}{2}][f,g]$. Similarly as for $G_{16}$ in the previous section we see that every $F \in R[\frac{1}{2}][L_1]^{G_4}$ can be written as $F=F_1+gF_2$ with $F_1, F_2 \in K[f,g^2] \cap R[\frac{1}{2}][L_1]=R[\frac{1}{2}][L_1]^{G_6}=R[\frac{1}{2}][f,h]$. Since $h \in R[\frac{1}{2}][f,g]$, we therefore get $F \in R[\frac{1}{2}][f,g]$, so we get indeed that $R[\frac{1}{2}][L_1]^{G_4}=R[\frac{1}{2}][f,g]$.

The second integral reflection representation $L_2$ of these groups is obtained by conjugating the matrices given above with the matrix 
\[ U=\begin{pmatrix} 2 & 1 \\ 0 & -1 \end{pmatrix}. \]
Then the rings of invariants over $K$ are obtained by replacing $f$ and $g$ in the invariant rings for $L_1$ above by the following two polynomials:
\begin{align*}
f' &\coloneqq \frac{1}{16}(U \cdot f)=px^4+4 \omega x^3y +2(\omega-1)x^2y^2-xy^3, \\
g' &\coloneqq \frac{1}{8}(U \cdot g)=
8x^6+16(\omega+2)x^5y+40(\omega+1)x^4y^2-20px^3y^3+20\omega x^2y^4+4(\omega-1)x^5y-y^6.
\end{align*}

Using Proposition \ref{AlgIndLM} with the lexicographic ordering with $y>x$ we get directly, that all the invariant rings $R[L_2]^{G_i}$ are polynomial rings; more precisely, we have $R[L_2]^{G_4}=R[f',g']$, $R[L_2]^{G_5}=R[f'^3,g']$, $R[L_2]^{G_6}=R[f,g'^2]$, and $R[L_2]^{G_7}=R[f'^3,g'^2]$.

Finally, for $G_6$ and $G_7$ there is a third integral representation $L_3$ that we need to consider. This is obtained by conjugating the matrices defining $L_1$ with
\[ U' \coloneqq \begin{pmatrix} 1+i & 1 \\ 0 & -1 \end{pmatrix}. \]
Since $\det(U') \in R[\frac{1}{2}]^\times$, the representations $L_1$ and $L_3$ are equivalent over $R[\frac{1}{2}]$. Since $R[L_1]^{G_6}$ and $R[L_1]^{G_7}$ are polynomial rings, the same therefore holds for $R[\frac{1}{2}][L_3]^{G_6}$ and $R[\frac{1}{2}][L_3]^{G_7}$. Moreover, $1+i$ is the only prime element in $R$ which divides $2$ and an explicit calculation shows that $f''=\frac{1}{4}(U' \cdot f)$ and $g''=\frac{1}{8}(U' \cdot g)$ are algebraically independent modulo $1+i$, so by Proposition \ref{AlgIndResidue} we get that $R[L_3]^{G_6}=R[f'',g''^2]$ and $R[L_3]^{G_7}=R[f''^3,g''^2]$.

\subsection{Summary}

The results of this section can be summarized in the following theorem:

\begin{theorem}
Let $n \in \{ 4, \ldots, 22 \}$, let $K \subseteq \mathbb{C}$ be the smallest number field over which $G_n$ can be defined, let $R$ be the ring of integers of $K$, and let $L$ be an integral reflection representation of $G_n$. Then $R[L]^{G_n}$ is isomorphic to a polynomial ring except if $n \in \{ 8,16,20 \}$ or $n=4$ and $L$ is the first of the two integral reflection representations of $G_4$.
\end{theorem}

\bibliographystyle{plainurl}

\bibliography{AInvPseudoIrred}

\begin{thebibliography}{10}

\bibitem{bh}
Winfried Bruns and Jürgen Herzog.
\newblock {\em Cohen-Macaulay rings}.
\newblock Cambridge University Press, Cambridge, 1993.

\bibitem{craig}
Maurice Craig.
\newblock A characterization of certain extreme forms.
\newblock {\em Illinois Journal of Mathematics}, 20:706--717, 1976.

\bibitem{singular}
Wolfram Decker, Gert-Martin Greuel, Gerhard Pfister, and Hans Schönemann.
\newblock {\sc Singular} {4-0-2} --- {A} computer algebra system for polynomial
  computations.
\newblock \url{http://www.singular.uni-kl.de}, 2015.

\bibitem{feit1}
Walter Feit.
\newblock Integral representations of weyl groups rationally equivalent to the
  reflection representation.
\newblock {\em Journal of Group Theory}, 1:213--218, 1998.

\bibitem{feit2}
Walter Feit.
\newblock Some integral representations of complex reflection groups.
\newblock {\em Journal of Algebra}, 260:138--153, 2003.

\bibitem{finvar}
Agnes~E. Heydtmann and Simon~A. King.
\newblock {\tt finvar.lib}. {A} {\sc singular} {4-0-2} library for computing
  polynomial invariants of finite matrix groups and generators of related
  varieties, 2014.

\bibitem{kemperrefl}
Gregor Kemper.
\newblock Calculating invariant rings of finite groups over arbitrary fields.
\newblock {\em Journal of Symbolic Computation}, 21:351--366, 1996.

\bibitem{masterthesis}
David Mundelius.
\newblock Rings of invariants of irreducible reflection groups over rings of
  algebraic integers.
\newblock Master's thesis, Technische Universität München, Germany, 2017.

\bibitem{ainvpseudo}
David Mundelius.
\newblock Arithmetic invariants of pseudoreflection groups and regular graded
  algebras.
\newblock {\em Journal of algebra}, 595:244--259, 2022.

\bibitem{st}
Geoffrey~Colin Shephard and John~Arthur Todd.
\newblock Finite unitary reflection groups.
\newblock {\em Canadian Journal of Mathematics}, 6:274--304, 1954.

\end{thebibliography}

\end{document}